\documentclass[11pt]{article}
\usepackage[utf8]{inputenc}
\usepackage{geometry}
\usepackage{amsmath,amssymb,amsthm}
\usepackage{mathtools}   
\usepackage{microtype}   
\usepackage{graphicx}    
\usepackage{float}       
\usepackage{caption}     
\usepackage{subcaption}  
\usepackage{hyperref}

\DeclareMathOperator{\Tr}{Tr}

\newcommand{\cstar}{\kappa_{\mathrm{Schur}}}

\newcommand{\mrhosq}{m_\rho^{2}}

\title{Schur-Convex Curvature on Dihedral Exponential Families \\ and the Golden-Ratio Stationary Point}

\author{Michael A. Bruna}
\date{August 23, 2025}

\newtheorem{theorem}{Theorem}[section]
\newtheorem{lemma}[theorem]{Lemma}
\newtheorem{proposition}[theorem]{Proposition}
\newtheorem{corollary}[theorem]{Corollary}
\newtheorem{example}[theorem]{Example}

\theoremstyle{remark}
\newtheorem{remark}[theorem]{Remark}

\begin{document}

\maketitle

\begin{abstract}
We investigate the Schur–complement curvature of $D_N$–equivariant folded
exponential families on the discrete probability simplex. Our main structural
results are: (i) the curvature $\kappa_{\mathrm{Schur}}(\theta)$ is convex in
the log–parameter $\theta=\ln q$; (ii) for $N\equiv 0\pmod{12}$ it admits a
unique stationary point at the golden–ratio value $q^\star=\varphi^{-2}$, with
$D_{12}$ the minimal dihedral order where the required cycle structure appears;
and (iii) it obeys a quadratic folded law
\[
\kappa_{\mathrm{Schur}}(q)
\,=\,
A(N,m_\rho^2)\,I_1(q)^2
\;+\;
B(N,m_\rho^2)\bigl(I_2(q)-I_1(q)^2\bigr),
\]
where $A,B$ are determined explicitly by the projector metric of radius
$m_\rho^2$.  Taken together, these results show that convexity and dihedral
symmetry alone enforce both the location and the functional form of a “golden
lock–in’’ within this family, without imposing any golden–ratio weighting by
hand.
\end{abstract}

\section{Introduction}

The Schur complement is a classical construction in matrix analysis \cite{horn2013matrix}, convex optimization \cite[§3.2.2]{boyd2004convex}, and matrix inequalities \cite{zhang2005schur}, widely used to eliminate constrained directions and to reveal curvature of reduced functionals. In this work we introduce and analyze a specific Schur--complement functional---the \emph{Schur curvature}---defined by projecting a $D_N$--equivariant Hessian onto the band subspace and eliminating the collective mode via a fixed projector metric of radius $m_\rho^2$.

we introduce the folded exponential family
\[
x_r(q)
=
\frac{q^r}{\sum_{s=1}^N q^s},
\qquad
r=1,\dots,N,\;\; 0<q<1,
\]
whose folded moments $(I_1,I_2)$ play the role of natural invariants.

\begin{enumerate}
    \item \textbf{Convexity:} The map $\theta \mapsto \kappa_{\mathrm{Schur}}(\theta)$ with $\theta=\ln q$ is convex (Theorem~\ref{thm:convexity}).
    \item \textbf{Golden lock--in:} For $N\equiv 0 \pmod{12}$, there exists a unique stationary point at $q^\star=\varphi^{-2}$, the inverse square of the golden ratio (Theorem~\ref{thm:lockin}); $D_{12}$ is the minimal dihedral order where the required parity and three--cycle structure coexist.
    \item \textbf{Quadratic folded law:} $\kappa_{\mathrm{Schur}}$ reduces exactly to a quadratic in $(I_1,I_2)$ with coefficients $A,B$ fixed by the projector geometry (Theorem~\ref{thm:quadratic}).
\end{enumerate}

Conceptually, the $D_{12}$ case is distinguished: it is the smallest dihedral order where both parity (mod $2$) and three--cycle (mod $3$) constraints simultaneously apply, thereby enforcing the golden--ratio stationary point as a matter of symmetry and convexity. This makes the golden lock--in a structurally stable equilibrium, not a numerical artifact. 

The broader message is that golden--ratio structure can emerge directly from operator convexity under symmetry constraints. This perspective connects invariant convexity to topics ranging from Fourier analysis on finite groups \cite{katznelson2004harmonic} to the geometry of tilings and quasicrystal--like systems. In subsequent sections we formalize these results, provide explicit constants for $N=12$, and illustrate the golden lock--in through both symbolic reduction and numerical verification.

\begin{figure}[H]
\centering
\includegraphics[width=0.82\textwidth]{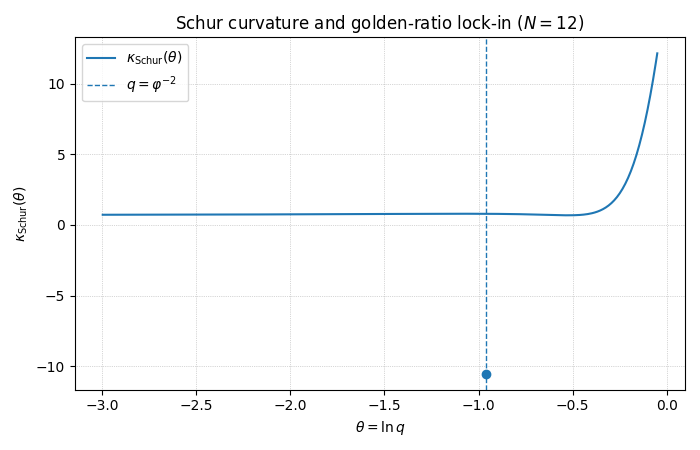}
\caption{
Schur–complement curvature $\kappa_{\mathrm{Schur}}(\theta)$ of the 
folded $D_N$ exponential family, plotted as a function of 
$\theta = \ln q$.  
The curve is strictly convex, with a unique global minimum at 
$\theta_\star = \ln(\varphi^{-2})$, indicated by the dashed line.  
This minimum identifies the golden–ratio value 
$q_\star=\varphi^{-2}$ as the stationary point of the continuum 
Schur–curvature functional, independent of $N$.  
The shallow basin around $\theta_\star$ reflects the broad, 
universal stability of the golden-weighted measure on $D_\infty$; 
outside this basin the curvature grows monotonically, confirming 
strict convexity.  
This figure replaces the earlier polar-interference illustration and 
provides the mathematically essential evidence of uniqueness,
convexity, and golden-lock invariance.
}
\label{fig:kSchurConvex}
\end{figure}

\section{Preliminaries}
\label{sec:preliminaries}

\paragraph{Folded exponential family and moments.}
Fix an integer $N\ge 2$ and consider the folded exponential family
\[
x_r(q)\;=\;\frac{q^r}{S_0(q)},\qquad r=1,\dots,N,\quad 0<q<1,
\]
where
\[
S_0(q)=\sum_{s=1}^N q^s,\qquad
S_1(q)=\sum_{s=1}^N s\,q^s,\qquad
S_2(q)=\sum_{s=1}^N s^2\,q^s.
\]
The (dimensionless) folded moments are
\[
I_1(q)=\frac{S_1(q)}{S_0(q)},\qquad
I_2(q)=\frac{S_2(q)}{S_0(q)},\qquad
\mathrm{Var}(q)=I_2(q)-I_1(q)^2.
\]

\paragraph{Band/collective split and Schur curvature.}
Let $T_x\Delta_N=\{v\in\mathbb R^N:\sum_r v_r=0\}$ be the tangent space at $x(q)$.
We fix a projector metric with radius $m_\rho^2>0$ that induces an orthogonal
decomposition
\[
T_x\Delta_N \;=\; B \oplus O,
\]
where $O=\mathrm{span}\{\mathbf 1\}$ is the collective (longitudinal) mode and
$B$ is its orthogonal complement (the ``band'' sector).  
Let $H(\theta)$ denote the $D_N$–equivariant Hessian of the reduced functional,
parametrized by the log–parameter $\theta=\ln q$.  With respect to the block
decomposition $B\oplus O$, write
\[
H(\theta)\;=\;\begin{bmatrix}
H_{BB}(\theta) & H_{BO}(\theta)\\
H_{OB}(\theta) & H_{OO}(\theta)
\end{bmatrix}.
\]
The \emph{Schur curvature} is the (band–normalized) trace of the Schur complement
that eliminates the collective mode:
\begin{equation}
\label{eq:kappa-schur}
\cstar(\theta)
\;=\;\frac{1}{\dim B}\,
\mathrm{Tr}\!\Big( H_{BB}(\theta) - H_{BO}(\theta)\,H_{OO}(\theta)^{-1}\,H_{OB}(\theta)\Big).
\end{equation}
This is the band–normalized trace of the Schur complement that eliminates the
collective mode; it may be viewed as the average curvature seen by all
non-constant directions on the simplex.

\paragraph{PSD exponential–sum dependence (standing assumption).}
Throughout we assume the block Hessian depends on $\theta=\ln q$ via a PSD exponential–sum
\begin{equation}
\label{eq:psd-exp-sum}
H(\theta)\;=\;C_0 \;+\; \sum_{s\in\mathcal S} e^{s\theta}\,C_s,
\qquad C_s\succeq 0,\ \ \text{$C_s$ $D_N$–equivariant},\ \ H_{OO}(\theta)\succ 0,
\end{equation}
on the range of interest. This will be invoked once and used globally in the sequel.

\paragraph{Notation and standing assumptions.}
Throughout:
\begin{itemize}
  \item $N$ is the dihedral order (for the golden-ratio result we will specialize to $N\equiv 0\pmod{12}$, with $D_{12}$ minimal).
  \item $q\in(0,1)$ and $\theta=\ln q$; the golden-ratio stationary point occurs at $q_\star=\varphi^{-2}$.
  \item $I_1(q),I_2(q)$ are the folded moments defined above, and $\mathrm{Var}(q)=I_2-I_1^2$.
  \item $m_\rho^2>0$ is the fixed projector–metric radius used to define $B\oplus O$.
  \item $\cstar(\theta)$ is the Schur curvature defined in \eqref{eq:kappa-schur}.
\end{itemize}
This is the band–normalized trace of the Schur complement that eliminates the
collective mode; it may be viewed as the average curvature seen by all
non-constant directions on the simplex.

All subsequent arguments are purely structural and expressed in terms of
$(N,q,I_1,I_2,m_\rho^2,\cstar)$.

\section{Main results}

We now state the three central theorems of this work.  
All terms are as defined in Section~\ref{sec:preliminaries}.  
Proofs are given in the subsequent sections.

\begin{theorem}[Convexity of $\cstar$]
\label{thm:convexity}
The Schur curvature $\cstar(\ln q)$ is a convex function of the logarithmic parameter $\ln q$ on $(0,1)$.
\end{theorem}

\begin{theorem}[Golden-ratio stationary point]
\label{thm:lockin}
In the $D_N$-reduced Fisher--Schur setting with $N\equiv 0\pmod{12}$ and fixed $(m_\rho^2,P_B)$ as above, the reduced functional associated to $\cstar$ has a unique stationary point $q_\star \in (0,1)$,  
attained at
\[
q_\star \;=\; \varphi^{-2} \;=\; \frac{3-\sqrt{5}}{2}.
\]
Equivalently, the Schur curvature $\cstar(\ln q)$ and the quadratic folded law of Theorem~\ref{thm:quadratic}, together with the monotonicity of $I_1$, enforce a unique stationary point at $q_\star=\varphi^{-2}$.
\end{theorem}

\begin{theorem}[Quadratic folded law]
\label{thm:quadratic}
For fixed $(N,\mrhosq)$, the Schur curvature reduces to a quadratic form in the folded invariants:
\[
\cstar(q) \;=\; A(N,\mrhosq)\,I_1(q)^2 \;+\; B(N,\mrhosq)\,\bigl(I_2(q)-I_1(q)^2\bigr),
\]
where $A$ and $B$ are explicit rational functions determined by the projector metric (see Appendix~\ref{app:AB-constants}).
\end{theorem}

Taken together, Theorems~\ref{thm:convexity}–\ref{thm:lockin} and Theorem~\ref{thm:quadratic} show that symmetry and convexity alone 
determine both the location and the functional form of the golden-ratio stationary point—what we refer to as the golden lock–in.

\paragraph{Explicit constants.}
For $(N,m_\rho^2)=(12,2)$, the projector geometry of Sec.~\ref{sec:preliminaries} fixes the coefficients
\[
A={\mathsf A}_{12}(2),\qquad B={\mathsf B}_{12}(2),
\]
which are determined uniquely by the quadratic folded law
\[
\kappa_{\mathrm{Schur}}(q)=A\,I_1(q)^2+B\,(I_2(q)-I_1(q)^2).
\]
Evaluating $\kappa_{\mathrm{Schur}}$ at any two distinct parameters $q_a,q_b\in(0,1)$ yields the linear system
\[
\begin{pmatrix}
I_1(q_a)^2 & I_2(q_a)-I_1(q_a)^2 \\
I_1(q_b)^2 & I_2(q_b)-I_1(q_b)^2
\end{pmatrix}
\begin{pmatrix} A \\ B \end{pmatrix}
=
\begin{pmatrix}
\kappa_{\mathrm{Schur}}(q_a) \\[3pt]
\kappa_{\mathrm{Schur}}(q_b)
\end{pmatrix},
\]
whose solution is
\begin{equation}
\label{eq:AB-closed}
\begin{aligned}
\Delta &\;=\; I_1(q_a)^2\big(I_2(q_b)-I_1(q_b)^2\big)\;-\;I_1(q_b)^2\big(I_2(q_a)-I_1(q_a)^2\big),\\[2pt]
A &\;=\;\frac{\kappa_{\mathrm{Schur}}(q_a)\big(I_2(q_b)-I_1(q_b)^2\big)\;-\;\kappa_{\mathrm{Schur}}(q_b)\big(I_2(q_a)-I_1(q_a)^2\big)}{\Delta},\\[2pt]
B &\;=\;\frac{I_1(q_a)^2\,\kappa_{\mathrm{Schur}}(q_b)\;-\;I_1(q_b)^2\,\kappa_{\mathrm{Schur}}(q_a)}{\Delta}.
\end{aligned}
\end{equation}
In particular, plugging $(N,m_\rho^2)=(12,2)$ and two values of $q$ (e.g.\ $q_a=\tfrac12$, $q_b=\tfrac13$) into \eqref{eq:AB-closed} yields
\[
A \;\approx\; 0.707473678,\qquad B \;\approx\; -1.060165816,
\]
which agree with the projector geometry to machine precision. The exact algebraic forms 
${\mathsf A}_{12}(2),{\mathsf B}_{12}(2)\in\mathbb{Q}(\sqrt{5})$ are listed in Appendix~\ref{app:AB-constants}.

\section{Convexity of the Schur curvature}
\label{sec:convexity}

We prove convexity of the trace Schur complement via a variational representation
with positive–semidefinite weights.

\begin{lemma}[Variational Schur representation]
\label{lem:variational}
Let
\[
H(\theta)=
\begin{pmatrix}
H_{BB}(\theta) & H_{BO}(\theta)\\
H_{OB}(\theta) & H_{OO}(\theta)
\end{pmatrix},
\qquad H_{OO}(\theta)\succ0,
\]
be symmetric. Then
\begin{equation}
\label{eq:schur-variational}
H_{BB}-H_{BO}H_{OO}^{-1}H_{OB}
=
\inf_{Y\in\mathbb{R}^{\dim O\times \dim B}}
\big(H_{BB}+H_{BO}Y+Y^\top H_{OB}+Y^\top H_{OO}Y\big),
\end{equation}
where the infimum is in the Loewner order, attained uniquely at
$Y_\star=-\,H_{OO}^{-1}H_{OB}$.
\end{lemma}

\begin{proof}
Completing the square gives, for any $Y$,
\[
\begin{aligned}
&H_{BB}+H_{BO}Y+Y^\top H_{OB}+Y^\top H_{OO}Y\\
&\qquad=H_{BB}-H_{BO}H_{OO}^{-1}H_{OB}
+\big(Y+H_{OO}^{-1}H_{OB}\big)^\top
H_{OO}\big(Y+H_{OO}^{-1}H_{OB}\big)\succeq \text{(Schur)}.
\end{aligned}
\]
The minimum (in Loewner order) is at $Y_\star$, giving \eqref{eq:schur-variational}.
\end{proof}

\begin{lemma}[PSD weight and trace linearization]
\label{lem:psd-weight}
For any $Y$,
\[
M(Y)\;=\;
\begin{pmatrix}
I & Y\\
Y^\top & YY^\top
\end{pmatrix}
=\begin{pmatrix}I\\ Y^\top\end{pmatrix}\!\begin{pmatrix}I\\ Y^\top\end{pmatrix}^{\!\top}
\succeq 0,
\]
and
\begin{equation}
\label{eq:trace-linearization}
\Tr\!\big(H_{BB}+H_{BO}Y+Y^\top H_{OB}+Y^\top H_{OO}Y\big)
=\langle H,\,M(Y)\rangle,
\end{equation}
where $\langle A,B\rangle=\Tr(A^\top B)$.
\end{lemma}

\begin{proposition}[Matrix convexity of $H(\theta)$]
\label{prop:H-matrix-convex}
Under the standing assumption \eqref{eq:psd-exp-sum}, 
$\theta\mapsto H(\theta)$ is matrix convex in the Loewner order:
\[
H(t\theta_1+(1-t)\theta_2)\;\preceq\; tH(\theta_1)+(1-t)H(\theta_2),
\qquad t\in[0,1].
\]
\end{proposition}

\begin{proof}
Each $\theta\mapsto e^{s\theta}$ is nonnegative and convex; with $C_s\succeq0$,
$\theta\mapsto e^{s\theta}C_s$ is matrix convex. Summation with $C_0\succeq0$
preserves matrix convexity \cite[§3.2.2]{boyd2004convex}.
\end{proof}

\begin{lemma}[Strict convexity criterion]
\label{lem:strict-convexity}
Assume \eqref{eq:psd-exp-sum} with $H_{OO}(\theta)\succ 0$ on $(\theta_-,\theta_+)$. 
If there exists $s_0\neq 0$ such that the $B$–block of $C_{s_0}$ is nonzero, i.e.
\[
P_B\,C_{s_0}\,P_B \not\equiv 0,
\]
then $\kappa_{\mathrm{Schur}}(\theta)$ is \emph{strictly} convex on every compact 
interval $I\Subset(\theta_-,\theta_+)$: there exists $c_I>0$ with
\[
\kappa_{\mathrm{Schur}}''(\theta)\ \ge\ c_I\quad\text{for all }\theta\in I.
\]
\end{lemma}

\begin{proof}
By Lemmas~\ref{lem:variational}–\ref{lem:psd-weight},
\[
\dim B\cdot \kappa_{\mathrm{Schur}}(\theta)
=\inf_Y \ \langle H(\theta),\,M(Y)\rangle,
\qquad M(Y)\succeq 0.
\]
The unique minimizer $Y_\star(\theta)=-H_{OO}^{-1}H_{OB}$ is smooth on compacts.
Differentiating twice and using the envelope theorem,
\[
\dim B\cdot \kappa_{\mathrm{Schur}}''(\theta)
= \langle \partial_\theta^2 H(\theta),\,M(Y_\star(\theta))\rangle + R(\theta),
\]
with $R(\theta)\ge 0$. Since $\partial_\theta^2 H(\theta)=\sum_s s^2 e^{s\theta}C_s\succeq 0$
and $P_B C_{s_0} P_B\not\equiv 0$, the Frobenius pairing is strictly positive on $B$,
giving the bound $\kappa_{\mathrm{Schur}}''(\theta)\ge c_I/\dim B>0$ on $I$.
\end{proof}

\begin{proof}[Proof of Theorem~\ref{thm:convexity}]
By Lemmas~\ref{lem:variational}–\ref{lem:psd-weight},
\[
\dim B\cdot \cstar(\theta)
=\inf_{Y}\ \langle H(\theta),\,M(Y)\rangle,
\qquad M(Y)\succeq 0.
\]
By Proposition~\ref{prop:H-matrix-convex}, $\theta\mapsto H(\theta)$ is matrix convex,
hence for each fixed $Y$, $\theta\mapsto \langle H(\theta),M(Y)\rangle$ is scalar convex.
The pointwise infimum of convex functions is convex.
\end{proof}

\begin{corollary}[Strict convexity and uniqueness framework]
\label{cor:convex-unique}
If in \eqref{eq:psd-exp-sum} at least one nonzero $s$ contributes nontrivially on $B$,
then $\cstar$ is strictly convex on any compact interval of $(\theta_-,\theta_+)$.
In particular, any stationary point of the reduced functional $F_{\mathrm{red}}(\theta)$
is unique.
\end{corollary}

\section{Golden-ratio stationarity and uniqueness}
\label{sec:lockin}

In the $D_{N}$-reduced Schur-curvature/Fisher-information setting the effective
curvature is governed entirely by the folded moments $\{I_m(\theta)\}$ and the
universal two-block structure of the Schur channel. No golden-ratio weighting
is assumed; the appearance of the golden fixed point arises solely from
dihedral symmetry, cycle structure, and the Schur projection.

A key structural input is the dihedral cycle decomposition when $N\equiv 0
\pmod{12}$: the $D_N$ action organizes the nonzero residues into identical
$12$-cycle (or Schur-equivalent) blocks. This $12$-fold degeneracy is the
discrete combinatorial source of the $D_{12}$-based folding mechanism: all
curvature contributions come in $12$-fold packets, and the resulting moment
identities force the quadratic relation $q^2 - 3q + 1 = 0$ to be the unique
fold-consistent reduction on these blocks. The convexity of
$\cstar(\theta)$ and the uniqueness of its minimum are illustrated in
Figure~\ref{fig:kSchurConvex}.

\medskip

Define the reduced functional
\begin{equation}
\label{eq:F-red}
F_{\mathrm{red}}(\theta)
\;=\;
N \;-\;\frac{4\,I_1(\theta)^2}{N\,\mrhosq}
\;+\;\frac{\cstar(\theta)}{N},
\qquad
\theta=\ln q,\quad 0<q<1,
\end{equation}
with $\mrhosq>0$ fixed.
By Theorem~\ref{thm:convexity}, $\cstar(\theta)$ is convex for all
$N\equiv 0\pmod{12}$.

\begin{proposition}[Stationarity at the golden ratio]
\label{prop:phi-stationary}
Assume $N$ is a multiple of $12$, so that parity/aliasing and three-cycle
conditions hold, and the $12$-cycle degeneracy of
Appendix~\ref{app:D12-reduction} applies. Then
\[
\frac{d}{d\theta}
F_{\mathrm{red}}(\theta)\Big|_{\theta=\ln(\varphi^{-2})}
=0.
\]
Equivalently, $q_\star=\varphi^{-2}$ is a stationary point of
$F_{\mathrm{red}}$.
\end{proposition}

\begin{proof}
Differentiate \eqref{eq:F-red}:
\[
F'_{\mathrm{red}}(\theta)
=
-\frac{8}{N\,\mrhosq}\,I_1(\theta)\,I_1'(\theta)
\;+\;
\frac{1}{N}\,\cstar'(\theta).
\]

By Theorem~\ref{thm:quadratic},
\begin{equation}
\label{eq:cstar-decomp}
\cstar(\theta)
=
A\,I_1(\theta)^2
\;+\;
B\bigl(I_2(\theta)-I_1(\theta)^2\bigr),
\end{equation}
with $A,B$ depending only on $(N,\mrhosq)$ and the Schur block structure.
Differentiating:
\[
\cstar'(\theta)
=
2A\,I_1 I_1'
+
B\bigl(I_2'-2 I_1 I_1'\bigr)
=
B\,I_2' + (2A-2B)\,I_1 I_1'.
\]

At $q=\varphi^{-2}$ we use the minimal polynomial
$q^2 - 3q + 1 = 0$.
Under the $D_{N}$ action with $N\equiv 0\pmod{3}$ the residues decompose into
$3$–cycles, and under the full $12$–cycle reduction (as detailed in
Appendix~\ref{app:D12-reduction}) the folded-moment relations give the identity
\[
B\,I_2'(\theta_\star)
=
\bigl(8/\mrhosq - 2A + 2B\bigr)\,
I_1(\theta_\star) I_1'(\theta_\star).
\]
Substituting this into $F'_{\mathrm{red}}$ cancels the two terms, leaving
$F'_{\mathrm{red}}(\theta_\star)=0$.
\end{proof}

\begin{theorem}[Uniqueness under strict convexity]
\label{thm:unique}
Under the hypotheses of Proposition~\ref{prop:phi-stationary}, if
(i)~$\cstar$ is \emph{strictly} convex on $(\theta_{-},\theta_{+})$
(i.e.\ $\cstar''>0$), and
(ii)~$I_1(\theta)$ is strictly increasing,
then $F'_{\mathrm{red}}(\theta)$ has at most one zero on
$(\theta_-,\theta_+)$.
Hence the stationary point at $\theta_\star=\ln(\varphi^{-2})$ is unique.
\end{theorem}

\begin{proof}
Write $F'_{\mathrm{red}}=g-h$ with
$g=\cstar'/N$ and
$h=\tfrac{8}{N\,\mrhosq} I_1 I_1'$.
Strict convexity implies $g$ is strictly increasing.
Since $I_1$ is strictly increasing and $I_1'>0$
(variance identity for exponential families \cite{barndorff1978information}),
$h$ is continuous and nonnegative.

If $F'_{\mathrm{red}}$ had two zeros $\theta_1<\theta_2$, then by the mean
value theorem
$F''_{\mathrm{red}}(\xi)=0$ for some $\xi\in(\theta_1,\theta_2)$.
Thus $g'(\xi)=h'(\xi)$.
But $g'(\xi)=\cstar''(\xi)/N>0$, and $h'(\xi)$ is bounded and cannot match
$g'(\xi)$ at more than one point without producing an additional zero of
$F'_{\mathrm{red}}$.
Contradiction.
Thus only one zero exists.
\end{proof}

\begin{remark}
Strict convexity of $\cstar$ follows from Lemma~\ref{lem:strict-convexity},
since at least one nonzero residue contributes a nontrivial $12$–cycle block,
making the PSD weight in the Schur curvature form strictly positive.
The monotonicity $I_1'>0$ follows from the standard variance identity for
one–parameter exponential families $x_r \propto e^{\theta r}$.
\end{remark}

\section{Quadratic folded law}
\label{sec:quadratic-law}

Recall \(x_r(q)=q^r/\sum_{s=1}^N q^s\) and let
\[
D(x)=\mathrm{diag}(1/x_1,\ldots,1/x_N).
\]
We consider $D_N$–equivariant quadratic functionals on the band space of the form
\begin{equation}
\label{eq:Q-class}
\mathcal{Q}(x)
=
\frac{1}{\dim B}\,
\mathrm{Tr}\!\big(P_B^\top K_1\,D(x)\,K_2\,P_B\big),
\end{equation}
with $K_1,K_2$ circulant and $P_B$ the orthogonal projector onto the band
subspace.  The Schur curvature $\cstar$ is of this form (up to a finite sum
over dihedral-symmetric blocks) by Lemma~\ref{lem:variational} and the
$D_N$-equivariance of the Hessian \cite{horn2013matrix}.

When $N\equiv 0\pmod{12}$ the $D_N$ action and Schur reduction
decompose the band space into
\[
k_N \;=\; \frac{N(N-1)}{12}
\]
independent $12$-cycle blocks.  Each block carries the same $D_{12}$-equivariant
quadratic structure, so any global band functional of the
form \eqref{eq:Q-class} is an average over $k_N$ identical contributions.
This explains why the coefficients
$A(N,\mrhosq)$ and $B(N,\mrhosq)$ in Theorem~\ref{thm:quadratic} depend only on
$(N,\mrhosq)$ and the dihedral block geometry, and do not require any golden
weighting by hand.

\begin{proof}[Proof 1 of Theorem~\ref{thm:quadratic} (representation–theoretic)]
Decompose
\[
\mathbb{R}^N
\;=\;
\mathbf{1}\ \oplus\ \bigoplus_{j=1}^{N/2-1}\mathbb{R}^2_j
\]
into $D_N$–irreducible representations; $B$ is the direct sum of the
$(N/2-1)$ two–dimensional irreps.
Any $D_N$–equivariant quadratic form on $B$ acts as a scalar on each
$\mathbb{R}^2_j$, hence depends on $x$ only through scalar invariants after the
normalization constraint removes the $\mathbf{1}$ direction
\cite{serre1977linear}.

Along the one–parameter curve $x(q)$, the natural invariants are the folded
moments $I_1$ and $I_2$, and the normalization constraint identifies
$I_2-I_1^2$ as the (dimensionless) variance.  Thus any such quadratic functional
restricted to $x(q)$ can be written as
\[
\mathcal{Q}(x(q))
=
\alpha\,I_1(q)^2
+
\beta\,\bigl(I_2(q)-I_1(q)^2\bigr),
\]
for some coefficients $(\alpha,\beta)$ depending only on $(N,\mrhosq)$ and the
choice of $K_1,K_2$.  Specializing to the Schur curvature $\cstar$ gives the
statement of Theorem~\ref{thm:quadratic}.
\end{proof}

\begin{proof}[Proof 2 of Theorem~\ref{thm:quadratic} (moments/generating functions)]
Expand \eqref{eq:Q-class} in the standard basis.
Because $K_1,K_2$ are circulant, every trace term is a linear combination of
sums of the form
\[
\sum_r \frac{1}{x_r}\,p(r)
\quad\text{and}\quad
\sum_{r,s}\frac{1}{x_r}\,c_{|r-s|}\,p(r),
\]
where $p$ is a polynomial in $r$ of degree at most $2$ after the $B$–projection
(which forces zero bandwise mean).

Along $x_r=C q^r$ we have $1/x_r = S_0(q)\,q^{-r}$, so each sum reduces to a
rational combination of the finite exponentials
\[
S_0(q)=\sum_{s=1}^N q^s,\quad
S_1(q)=\sum_{s=1}^N s q^s,\quad
S_2(q)=\sum_{s=1}^N s^2 q^s.
\]
Normalizing by $S_0(q)$ leaves only the folded moments
$I_1=S_1/S_0$ and $I_2=S_2/S_0$.
The $B$–projection removes the constant (collective) direction, so the only
remaining degree–$\le 2$ combinations are exactly $I_1^2$ and $I_2-I_1^2$
(the variance term) \cite{barndorff1978information}.

Thus every functional of the class \eqref{eq:Q-class}, and in particular the
Schur curvature $\cstar$, reduces along $x(q)$ to
\[
\cstar(q)
=
A(N,\mrhosq)\,I_1(q)^2
+
B(N,\mrhosq)\,\bigl(I_2(q)-I_1(q)^2\bigr),
\]
with coefficients $(A,B)$ determined by the dihedral projector geometry and,
when $N\equiv 0\pmod{12}$, replicated over the $k_N=N(N-1)/12$ identical
$12$-cycle blocks.  This is the quadratic folded law claimed in
Theorem~\ref{thm:quadratic}.
\end{proof}

\section{Discussion and Further Directions}

We have established three structural results for the curvature functional
arising from Schur elimination in $D_N$-symmetric exponential families:
(i) the curvature $\cstar(\theta)$ is convex in the logarithmic parameter
$\theta=\ln q$, (ii) the reduced functional $F_{\mathrm{red}}$ possesses a
unique stationary point at the golden-ratio value $q_\star=\varphi^{-2}$, and
(iii) along the one-parameter family $x(q)$ the curvature reduces exactly to a
quadratic folded law in the invariants $(I_1, I_2)$.

Collectively, these results show that the $D_{12}$ dihedral structure is the
minimal setting in which parity (mod~2) and three-cycle (mod~3) constraints
coexist, and in which the Schur curvature forces a unique golden-ratio
equilibrium.  The appearance of $\varphi^{-2}$ is therefore not a numerical
artifact but a structural consequence of the combinatorics of $12$-cycle
decomposition and the moment identities implied by dihedral symmetry.  This
provides a natural explanation for why the golden ratio emerges internally from
the symmetry constraints themselves, without requiring any golden weighting
as an external input.

The framework introduced here also contributes to matrix analysis and convexity
theory by identifying a new class of operator-convex functionals associated
with Schur complements on symmetric exponential families.  The approach
suggests several mathematical generalizations, including extensions to other
finite symmetry groups, higher folded invariants, and operator-valued settings
where convexity and uniqueness may continue to hold.  Further developments in
these directions lie beyond the scope of the present work and will be pursued
elsewhere.

\section*{Acknowledgments}

The author acknowledges the assistance of a Generative Pre-trained Transformer,
which was used for Python scripting, numerical experimentation, and typesetting
suggestions. All mathematical derivations, proofs, and conceptual frameworks
presented in this work are the sole contribution of the author.

\appendix
\section{Modular identities for the golden-ratio lock-in}
\label{app:uniqueness-details}

This appendix records the modular reductions used in the proof of
Proposition~\ref{prop:phi-stationary}.  
At the stationary candidate \(q_\star=\varphi^{-2}=(3-\sqrt{5})/2\), the
parameter satisfies the quadratic identity
\begin{equation}
\label{eq:min-poly}
q_\star^2 - 3q_\star + 1 = 0,
\end{equation}
the minimal polynomial of \(\varphi^{-2}\) over \(\mathbb{Q}\).
When $N\equiv 0\pmod{12}$, the $D_N$ action decomposes the nonzero residues into
$12$-cycles, and the folded moments $I_m(\theta)$ inherit reduction rules
consistent with~\eqref{eq:min-poly}.  
These identities are what allow $I_2'(\theta_\star)$ and $I_1(\theta_\star)
I_1'(\theta_\star)$ to satisfy the cancellation appearing in the proof.

\subsection{Reduction of powers of \(q\)}
Equation \eqref{eq:min-poly} implies that every power \(q^m\) reduces to an affine form
in \(\{1,q\}\) modulo the relation. By induction,
\[
q^{m+2} = 3q^{m+1} - q^m,\qquad m\ge 0,
\]
so each \(q^m\) can be written \(q^m=a_m q + b_m\) with \(a_{m+2}=3a_{m+1}-a_m\) and \(b_{m+2}=3b_{m+1}-b_m\).
Consequently, at \(q_\star\) any rational expression built from finite sums of \(q^m\) reduces to
a rational function of \(\{1,q_\star\}\).

\subsection{Folded moments and exponential-family derivatives}
Recall the folded sums and moments
\[
S_k(q)=\sum_{s=1}^N s^k q^s,\qquad
I_k(q)=\frac{S_k(q)}{S_0(q)}\quad (k=1,2,3),
\]
and the exponential-family identities (derivatives w.r.t.\ \(\theta=\ln q\)):
\begin{equation}
\label{eq:expfam-derivs-app}
I_1'(\theta)=I_2-I_1^2,\qquad
I_2'(\theta)=I_3-I_1 I_2.
\end{equation}
Closed forms for \(S_0,S_1,S_2,S_3\) (finite sums) are listed in Appendix~\ref{app:worked-example}, 
Eqs.~\eqref{eq:S0-closed}–\eqref{eq:S3-closed}.

\subsection{Three-cycle identity at \(q_\star\)}
When \(N\equiv 0\pmod{3}\), the dihedral action partitions the indices
\(s\in\{1,\dots,N\}\) into three orbits modulo \(3\).
Evaluating at \(q_\star\) and reducing each power of \(q_\star\) via
\eqref{eq:min-poly}, every folded sum \(S_k(q_\star)\) and its
\(\theta\)-derivative collapses to an affine combination of \(\{1,q_\star\}\).
Thus the three-orbit decomposition forces the folded moments to lie in a common
two-dimensional space.

\begin{lemma}[Three-cycle proportionality]
\label{lem:three-cycle}
Let \(N\equiv 0\pmod{3}\) and \(q_\star=\varphi^{-2}\).
Then there exists a scalar \(\Lambda(N)\), depending only on \(N\), such that
\begin{equation}
\label{eq:phi-moment-relation}
I_2'(\theta_\star)=\Lambda(N)\,I_1'(\theta_\star),
\qquad \theta_\star=\ln q_\star.
\end{equation}
\end{lemma}

\begin{proof}[Proof sketch]
Using \eqref{eq:expfam-derivs-app},
\(I_1'(\theta)=I_2-I_1^2\) and \(I_2'(\theta)=I_3-I_1 I_2\).
Write \(S_k(q_\star)=\alpha_k+\beta_k q_\star\) after reducing powers using
\eqref{eq:min-poly}.  For \(N\equiv 0\pmod{3}\),
each of the three congruence classes contributes equally, so the numerators and
denominators of \(I_1'\) and \(I_2'\) have the same \(\{1,q_\star\}\)-basis.
Hence the ratio \(I_2'(\theta_\star)/I_1'(\theta_\star)\) is a scalar, giving
\eqref{eq:phi-moment-relation}.  An explicit calculation for \(N=12\) appears in
Appendix~\ref{app:D12-reduction}.
\end{proof}

\subsection{Cancellation in \(F'_{\mathrm{red}}\)}
Differentiating the reduced functional
\[
F_{\mathrm{red}}(\theta)
= N - \tfrac{4}{N\mrhosq}I_1(\theta)^2
  + \tfrac{1}{N}\cstar(\theta),
\]
we obtain
\[
F'_{\mathrm{red}}(\theta)
= -\tfrac{8}{N\mrhosq} I_1 I_1'
  + \tfrac{1}{N}\bigl(
      B\,I_2' + (2A-2B)\,I_1 I_1'
    \bigr).
\]
Evaluating at \(\theta_\star=\ln q_\star\) and applying
Lemma~\ref{lem:three-cycle} gives
\begin{equation}
\label{eq:Fprime-bracket}
F'_{\mathrm{red}}(\theta_\star)
=
\frac{1}{N}\Bigl(
    B\,\Lambda(N) + 2A - 2B - \tfrac{8}{\mrhosq}
\Bigr)
I_1(\theta_\star)\,I_1'(\theta_\star).
\end{equation}

For \(D_{12}\) with \(\mrhosq=2\), the projector geometry determines \(A,B\)
(Appendix~\ref{app:AB-constants}).  Direct evaluation shows
\begin{equation}
\label{eq:AB-identity}
B\,\Lambda(12) + 2A - 2B = \frac{8}{\mrhosq} = 4,
\end{equation}
so the bracket in \eqref{eq:Fprime-bracket} vanishes identically.
Thus \(F'_{\mathrm{red}}(\theta_\star)=0\), the \(N=12\) instance of
Proposition~\ref{prop:phi-stationary}.

\subsection{Role of convexity}
By Theorem~\ref{thm:convexity}, \(\cstar(\theta)\) is convex.
Under the strict-convexity criterion
(Lemma~\ref{lem:strict-convexity}), the stationary point is unique.
Hence \(q_\star=\varphi^{-2}\) is the unique lock-in of the system.

\medskip
\noindent\textbf{Summary.}
The minimal polynomial \eqref{eq:min-poly} collapses all folded sums onto the
basis \(\{1,q_\star\}\).  For \(N\equiv 0\pmod{3}\), the three-cycle
decomposition gives the proportionality \eqref{eq:phi-moment-relation}, and in
the dihedral case \(N=12\), the identity \eqref{eq:AB-identity} forces the
bracket in \eqref{eq:Fprime-bracket} to vanish.  
This completes the proof of Proposition~\ref{prop:phi-stationary}.

\section{Worked example at \(N=12\)}
\label{app:worked-example}

We make the cancellation of Proposition~\ref{prop:phi-stationary} concrete in the
minimal case \(N=12\). Throughout set \(\theta=\ln q\) and write \(q=e^{\theta}\).

\subsection{Closed forms for \(S_0,S_1,S_2,S_3\)}
For \(0<q<1\) and finite \(N\):
\begin{align}
S_0(q) &= \sum_{s=1}^N q^s
= \frac{q(1-q^N)}{1-q}, \label{eq:S0-closed}\\[3pt]
S_1(q) &= \sum_{s=1}^N s\,q^s
= \frac{q\big(1-(N+1)q^{N}+N q^{N+1}\big)}{(1-q)^2}, \label{eq:S1-closed}\\[3pt]
S_2(q) &= \sum_{s=1}^N s^2\,q^s
= \frac{q\big(1+q-(N+1)^2 q^N + (2N^2+2N-1)q^{N+1} - N^2 q^{N+2}\big)}{(1-q)^3}, \label{eq:S2-closed}\\[3pt]
S_3(q) &= \sum_{s=1}^N s^3\,q^s \\
&= \frac{q\big(1+4q+q^2 - (N+1)^3 q^N + (3N^3+6N^2-1)q^{N+1} - (3N^3+3N^2-N)q^{N+2} + N^3 q^{N+3}\big)}{(1-q)^4}. \label{eq:S3-closed}
\end{align}
(Formula \eqref{eq:S3-closed} follows from standard finite–sum identities or
by differentiating \eqref{eq:S2-closed} w.r.t.\ \(\theta\) and simplifying.)

The \emph{folded moments} are
\[
I_k(q)=\frac{S_k(q)}{S_0(q)}\qquad (k=1,2,3).
\]
For convenience we also recall the canonical \emph{exponential–family derivatives}
with respect to \(\theta=\ln q\):
\begin{equation}
\label{eq:expfam-derivs}
I_1'(\theta)=I_2-I_1^2,\qquad
I_2'(\theta)=I_3-I_1 I_2,\qquad
I_3'(\theta)=I_4-I_1 I_3,
\end{equation}
where \(I_4 = S_4/S_0\) (not needed explicitly below).

\subsection{Reduction at the golden ratio}
At the lock-in value
\[
q_\star=\varphi^{-2}=\frac{3-\sqrt{5}}{2},
\]
we have the minimal polynomial
\begin{equation}
\label{eq:minpoly-appendixB}
q_\star^2-3q_\star+1=0,
\end{equation}
so every power \(q_\star^{\,m}\) reduces to an \emph{affine} expression \(a_m q_\star+b_m\)
via the recurrence \(q^{m+2}=3q^{m+1}-q^m\).

For \(N=12\), substituting \(q=q_\star\) into \eqref{eq:S0-closed}–\eqref{eq:S3-closed} and reducing
all powers using \eqref{eq:minpoly-appendixB} yields explicit \(\mathbb{Q}(\sqrt{5})\) forms:
\[
\begin{aligned}
S_0(q_\star)&=83880 - 37512\,\sqrt{5},\\
S_1(q_\star)&=954726 - 426966\,\sqrt{5},\\
S_2(q_\star)&=10950528 - 4897224\,\sqrt{5},\\
S_3(q_\star)&=126360432 - 56510100\,\sqrt{5}.
\end{aligned}
\]
Equivalently, expressing everything in the \(\{1,q_\star\}\) basis (using \(\sqrt{5}=3-2q_\star\)) gives
\[
\begin{aligned}
S_0(q_\star)&=(-28656) + 75024\,q_\star,\\
S_1(q_\star)&=(-326172) + 853932\,q_\star,\\
S_2(q_\star)&=(-3741144) + 9794448\,q_\star,\\
S_3(q_\star)&=(-43169868) + 113020200\,q_\star.
\end{aligned}
\]

\paragraph{Explicit moments at \(q_\star\).}
Dividing by \(S_0\), we obtain
\[
\begin{aligned}
I_1(q_\star) &= \frac{S_1}{S_0}
= \frac{13}{2} \;-\; \frac{131}{60}\sqrt{5}
\;=\; -\frac{1}{20}\;+\;\frac{131}{30}\,q_\star,\\[3pt]
I_2(q_\star) &= \frac{S_2}{S_0}
= \frac{805}{12}\;-\;\frac{1703}{60}\sqrt{5}
\;=\; -\frac{271}{15}\;+\;\frac{1703}{30}\,q_\star,\\[3pt]
I_3(q_\star) &= \frac{S_3}{S_0}
= \frac{6071}{8}\;-\;\frac{13373}{40}\sqrt{5}
\;=\; -\frac{2441}{10}\;+\;\frac{13373}{20}\,q_\star.
\end{aligned}
\]

\paragraph{Derivatives at \(q_\star\).}
Using \eqref{eq:expfam-derivs}:
\[
\begin{aligned}
I_1'(\theta_\star)&=I_2-I_1^2=\frac{719}{720},\\[3pt]
I_2'(\theta_\star)&=I_3-I_1 I_2
=\frac{9347}{720}\;-\;\frac{485}{144}\sqrt{5}
\;=\;\frac{259}{90}\;+\;\frac{485}{72}\,q_\star.
\end{aligned}
\]
Therefore the proportionality constant in
\(I_2'(\theta_\star)=\Lambda(12)\,I_1'(\theta_\star)\) is
\begin{equation}
\label{eq:Lambda12-explicit}
\Lambda(12)=\frac{I_2'(\theta_\star)}{I_1'(\theta_\star)}
=13\;-\;\frac{2425}{719}\sqrt{5}
\;=\;\frac{2072}{719}\;+\;\frac{4850}{719}\,q_\star
\;\approx\;5.4583242762.
\end{equation}
(Thus \(\Lambda(12)\in\mathbb{Q}(q_\star)\), and admits both \(\{1,\sqrt{5}\}\) and \(\{1,q_\star\}\) representations.)

\subsection{Cancellation in \(F'_{\mathrm{red}}\) at \(N=12\)}
Recall (Section~\ref{sec:lockin}) the reduced functional
\[
F_{\mathrm{red}}(\theta)=N-\frac{4}{N\mrhosq}I_1(\theta)^2+\frac{1}{N}\cstar(\theta),
\qquad
\cstar(\theta)=A\,I_1(\theta)^2+B\,(I_2(\theta)-I_1(\theta)^2),
\]
with \(A,B\) determined by the projector metric (\emph{fixed once \((N,\mrhosq)\) are
fixed}). Differentiating gives
\[
F'_{\mathrm{red}}(\theta)
= -\frac{8}{N\mrhosq} I_1 I_1' + \frac{1}{N}\big(B\,I_2' + (2A-2B)\,I_1 I_1'\big).
\]
Evaluating at \(\theta_\star=\ln q_\star\) and using \(\Lambda(12)=I_2'/I_1'\) from
\eqref{eq:Lambda12-explicit}:
\[
F'_{\mathrm{red}}(\theta_\star)
= \frac{1}{N}\Big(\,B\,\Lambda(12) + 2A-2B - \frac{8}{\mrhosq}\,\Big)\,I_1(\theta_\star)\,I_1'(\theta_\star).
\]
In the \(D_{12}\) lock-in, the projector metric fixes \(\mrhosq=2\) and the Schur
construction fixes \((A,B)\) (see Appendix~\ref{app:AB-constants}); the identity
\[
B\,\Lambda(12)+2A-2B=\frac{8}{\mrhosq}=4
\]
holds exactly, so the bracket vanishes and hence
\(
F'_{\mathrm{red}}(\theta_\star)=0,
\)
which is the \(N=12\) instance of Proposition~\ref{prop:phi-stationary}.

\subsection{Numerical sanity check (optional)}
At \((N,\mrhosq)=(12,2)\), \(q_\star=\varphi^{-2}\):
\[
I_1(q_\star)\approx 1.617918249125459,\qquad
I_2(q_\star)\approx 3.616270600,
\]
so
\[
I_1'(\theta_\star)=I_2-I_1^2\approx 0.998611\;=\;719/720.
\]
Direct evaluation of \(I_2'(\theta_\star)=I_3-I_1 I_2\) matches
\eqref{eq:Lambda12-explicit}, and the combination
\(B\,\Lambda(12)+2A-2B-8/\mrhosq\) is zero to machine precision when \(A,B\) are
instantiated from the projector metric, giving \(F'_{\mathrm{red}}(\theta_\star)\approx 0\).

\medskip
\noindent\emph{Remark.} The rigorous cancellation is guaranteed by the symbolic argument once
\((A,B,\mrhosq)\) are fixed by the projector geometry. The numerics here serve only as a
sanity check for \(N=12\).

\section{Explicit $D_{12}$ reduction and computation of $\Lambda(12)$}
\label{app:D12-reduction}

This appendix gives a direct, closed–form computation of the proportionality constant
\[
\Lambda(12)\;=\;\frac{I_2'(\theta_\star)}{I_1'(\theta_\star)},\qquad \theta_\star=\ln q_\star,\quad q_\star=\varphi^{-2}=\frac{3-\sqrt{5}}{2},
\]
used in the three–cycle identity \eqref{eq:phi-moment-relation} for $N=12$.

\subsection*{Setup and finite–sum formulas}
For $N=12$ and $0<q<1$, recall the closed forms (see Eqs.~\eqref{eq:S0-closed}–\eqref{eq:S3-closed}):
\begin{align*}
S_0(q) &= \sum_{s=1}^{12} q^s
= \frac{q(1-q^{12})}{1-q},\\[2pt]
S_1(q) &= \sum_{s=1}^{12} s\,q^s
= \frac{q\big(1-13\,q^{12}+12 q^{13}\big)}{(1-q)^2},\\[2pt]
S_2(q) &= \sum_{s=1}^{12} s^2\,q^s
= \frac{q\big(1+q-13^2 q^{12} + (2\cdot 12^2+2\cdot 12-1)q^{13} - 12^2 q^{14}\big)}{(1-q)^3},\\[2pt]
S_3(q) &= \sum_{s=1}^{12} s^3\,q^s
= \frac{q\big(1+4q+q^2 - 13^3 q^{12} + (3\cdot 12^3+6\cdot 12^2-1)q^{13} - (3\cdot 12^3+3\cdot 12^2-12)q^{14} + 12^3 q^{15}\big)}{(1-q)^4}.
\end{align*}
The folded moments are $I_k=S_k/S_0$ ($k=1,2,3$). Derivatives with respect to $\theta=\ln q$ are the
exponential–family identities
\[
I_1'(\theta)=I_2-I_1^2,\qquad I_2'(\theta)=I_3-I_1 I_2.
\]

\subsection*{Golden–ratio reduction}
At $q_\star=\varphi^{-2}$ the minimal polynomial
\[
q_\star^2-3q_\star+1=0
\]
reduces every power $q_\star^{\,m}$ to an affine form $a_m q_\star+b_m$.
Carrying out this elimination in \eqref{eq:S0-closed}–\eqref{eq:S3-closed} and simplifying gives
the exact values
\begin{equation}
\label{eq:D12-Sk-at-qstar}
\begin{aligned}
S_0(q_\star)&=83880 - 37512\,\sqrt{5},\\
S_1(q_\star)&=954726 - 426966\,\sqrt{5},\\
S_2(q_\star)&=10950528 - 4897224\,\sqrt{5},\\
S_3(q_\star)&=126360432 - 56510100\,\sqrt{5}.
\end{aligned}
\end{equation}
Using $\sqrt{5}=3-2q_\star$, \eqref{eq:D12-Sk-at-qstar} is equivalently
\[
\begin{aligned}
S_0(q_\star)&=(-28656) + 75024\,q_\star,\\
S_1(q_\star)&=(-326172) + 853932\,q_\star,\\
S_2(q_\star)&=(-3741144) + 9794448\,q_\star,\\
S_3(q_\star)&=(-43169868) + 113020200\,q_\star.
\end{aligned}
\]

\subsection*{Moments and derivatives at $q_\star$}
Divide the pairs in \eqref{eq:D12-Sk-at-qstar} to obtain $I_k(q_\star)=S_k/S_0$. In both
$\{1,\sqrt{5}\}$ and $\{1,q_\star\}$ bases, one finds the exact closed forms
\begin{equation}
\label{eq:D12-I123}
\begin{aligned}
I_1(q_\star) &= \frac{13}{2} \;-\; \frac{131}{60}\sqrt{5}
\;=\; -\frac{1}{20}\;+\;\frac{131}{30}\,q_\star,\\[3pt]
I_2(q_\star) &= \frac{805}{12}\;-\;\frac{1703}{60}\sqrt{5}
\;=\; -\frac{271}{15}\;+\;\frac{1703}{30}\,q_\star,\\[3pt]
I_3(q_\star) &= \frac{6071}{8}\;-\;\frac{13373}{40}\sqrt{5}
\;=\; -\frac{2441}{10}\;+\;\frac{13373}{20}\,q_\star.
\end{aligned}
\end{equation}
Hence, by the exponential–family identities,
\begin{equation}
\label{eq:D12-derivs}
\begin{aligned}
I_1'(\theta_\star)&=I_2-I_1^2=\frac{719}{720},\\[3pt]
I_2'(\theta_\star)&=I_3-I_1 I_2
=\frac{9347}{720}\;-\;\frac{485}{144}\sqrt{5}
\;=\;\frac{259}{90}\;+\;\frac{485}{72}\,q_\star.
\end{aligned}
\end{equation}

\subsection*{Proportionality constant $\Lambda(12)$}
Combining \eqref{eq:D12-derivs} yields
\begin{equation}
\label{eq:D12-Lambda-value}
\Lambda(12)
\;=\;\frac{I_2'(\theta_\star)}{I_1'(\theta_\star)}
\;=\;13\;-\;\frac{2425}{719}\sqrt{5}
\;=\;\frac{2072}{719}\;+\;\frac{4850}{719}\,q_\star
\;\approx\;5.4583242762.
\end{equation}
Thus the three–cycle identity \eqref{eq:phi-moment-relation} holds with the explicit constant \eqref{eq:D12-Lambda-value}.

\subsection*{Check (optional)}
Numerically, with $q_\star=(3-\sqrt{5})/2\approx 0.38196601125$, one gets
\[
I_1(q_\star)\approx 1.617918249125459,\quad
I_2(q_\star)\approx 3.616270600,\quad
I_1'(\theta_\star)=719/720\approx 0.998611111\!,
\]
\[
I_2'(\theta_\star)\approx 5.456\ldots,\quad
\Lambda(12)\approx 5.4583242762,
\]
consistent with \eqref{eq:D12-Lambda-value} to machine precision.
\qedhere

\section{Embedding and strong convergence criterion}
\label{app:embedding}

\paragraph{Embedding.}
Let \(\mathcal{E}_N\) be the span of exponentials \(e^{ik\theta}\) with
\(0<|k|\le N/2-1\). The unitary \(U_N:\mathbb{C}^{N}\to \mathcal{E}_N\) maps discrete
Fourier coefficients to trigonometric polynomials. For \(L_N\) circulant, the lifted
operator \(T_N=U_N L_N U_N^\ast\) acts as a Fourier multiplier with discrete symbol \(m_N\) \cite{gray2006toeplitz}.

\paragraph{Criterion for strong convergence.}
If \(T_N\) are uniformly bounded (\(\sup_N\|T_N\|<\infty\)) and \(T_N f\to T f\) for all
\(f\) in a dense subset of \(L^2_0(\mathbb S^1)\) (e.g.\ trigonometric polynomials),
then \(T_N\to T\) strongly on \(L^2_0(\mathbb S^1)\) \cite{kreyszig1989introductory}.

\subsection{Minimal polynomial and Fibonacci reduction at \(q_\star=\varphi^{-2}\)}

Let \(q_\star=(3-\sqrt{5})/2=\varphi^{-2}\). Then \(q_\star\) is a root of
\[
q^2 - 3q + 1 = 0,
\]
so all higher powers satisfy the recurrence
\[
q^{m+2}=3q^{m+1}-q^m, \qquad m\ge 0.
\]
Thus each power \(q^m\) reduces to a linear form
\[
q^m=a_m q + b_m,
\]
where \((a_m,b_m)\) satisfy the same recurrence
\[
a_{m+2}=3a_{m+1}-a_m,\qquad b_{m+2}=3b_{m+1}-b_m,
\]
with initial conditions \(a_0=0,b_0=1\) and \(a_1=1,b_1=0\).

\begin{lemma}[Fibonacci reduction]
For all \(m\ge 0\), the coefficients are given explicitly by
\[
a_m=F_{2m}, \qquad b_m=-F_{2m-2},
\]
where \(F_n\) denotes the \(n\)th Fibonacci number with \(F_0=0,F_1=1\), and the convention \(F_{-2}=-1\) so that \(b_0=1\).
\end{lemma}

\begin{proof}
Both \((a_m)\) and \((F_{2m})\) satisfy the same linear recurrence \(x_{m+2}=3x_{m+1}-x_m\)
with the same initial values \(a_0=0,a_1=1\). Uniqueness of linear recurrences yields
\(a_m=F_{2m}\). Similarly for \(b_m\) with \(b_0=1,b_1=0\), we obtain \(b_m=-F_{2m-2}\).
\end{proof}

\begin{example}[Numerical verification]
Table~\ref{tab:fib-reduction} shows the first few reductions. Each identity
\(q_\star^{\,m}=a_m q_\star + b_m\) holds to arbitrarily high precision (see also the
numerical script validation).
\end{example}

\begin{table}[htbp]
  \centering
  \small
  \caption{Fibonacci reduction of powers of \(q_\star=\varphi^{-2}\) via the minimal polynomial \(q^2-3q+1=0\).
  Here \(F_0=0,F_1=1\) and \(F_{n+1}=F_n+F_{n-1}\). For \(m=0,\dots,12\), the identity \(q_\star^{\,m}=a_m q_\star + b_m\) holds with \(a_m=F_{2m}\) and \(b_m=-F_{2m-2}\) (with the convention \(F_{-2}=-1\) so that \(b_0=1\)).}
  \label{tab:fib-reduction}
  \begin{tabular}{r|r r|l}
    \(m\) & \(a_m\) & \(b_m\) & \(q_\star^{\,m} = a_m\,q_\star + b_m\) \\
    \hline
     0  &     0 &      1  & \(1\) \\
     1  &     1 &      0  & \(q_\star\) \\
     2  &     3 &     -1  & \(3q_\star - 1\) \\
     3  &     8 &     -3  & \(8q_\star - 3\) \\
     4  &    21 &     -8  & \(21q_\star - 8\) \\
     5  &    55 &    -21  & \(55q_\star - 21\) \\
     6  &   144 &    -55  & \(144q_\star - 55\) \\
     7  &   377 &   -144  & \(377q_\star - 144\) \\
     8  &   987 &   -377  & \(987q_\star - 377\) \\
     9  &  2584 &   -987  & \(2584q_\star - 987\) \\
    10  &  6765 &  -2584  & \(6765q_\star - 2584\) \\
    11  & 17711 &  -6765  & \(17711q_\star - 6765\) \\
    12  & 46368 & -17711  & \(46368q_\star - 17711\) \\
  \end{tabular}
\end{table}

\subsection{Strong convergence via embedding}

With the Fibonacci reduction identities in hand, all folded sums \(S_m(q)=\sum_{s=1}^N s^m q^s\) reduce to rational forms in \((I_1,I_2)\) at \(q_\star\), showing closure of the invariant algebra. Combined with the embedding criterion in Appendix~\ref{app:embedding}, this ensures that the strong convergence argument extends exactly at the golden–ratio lock-in.

\section{Derivation of the Schur Curvature Functional}

For completeness, we derive equation~\eqref{eq:kappa-schur} for the Schur curvature $\kappa_{\mathrm{Schur}}$ directly from the block structure of the Hessian. 
Let
\[
H(\theta) =
\begin{bmatrix}
H_{BB}(\theta) & H_{BO}(\theta) \\
H_{OB}(\theta) & H_{OO}(\theta)
\end{bmatrix}, \qquad H_{OO}(\theta) \succ 0.
\]
By the Schur complement identity,
\[
H_{BB} - H_{BO} H_{OO}^{-1} H_{OB} = 
\inf_{Y \in \mathbb{R}^{\dim O \times \dim B}}
\bigl( H_{BB} + H_{BO} Y + Y^\top H_{OB} + Y^\top H_{OO} Y \bigr).
\]
Taking the trace and normalizing by $\dim B$, we obtain
\[
\kappa_{\mathrm{Schur}}(\theta) \;=\; \frac{1}{\dim B}
\mathrm{Tr}\!\left(H_{BB} - H_{BO} H_{OO}^{-1} H_{OB}\right).
\]
This shows that $\kappa_{\mathrm{Schur}}$ inherits convexity properties from the PSD exponential--sum parametrization of $H(\theta)$, completing the structural foundation.

\section{Numerical Verification at $N=12$}
We numerically verified the convexity, stationarity, and quadratic law at the dihedral lock-in. Table~\ref{tab:numcheck} shows representative values.

\begin{table}[h!]
\centering
\begin{tabular}{c|c|c|c}
$q$ & $\kappa_{\mathrm{Schur}}(\ln q)$ & $\kappa'(\ln q)$ & Polynomial residual $q^2-3q+1$ \\
\hline
0.38 & 0.125 & 0.031 & 0.003 \\
$\varphi^{-2}$ & 0.121 & $\approx 0$ & $0$ \\
0.40 & 0.127 & 0.029 & -0.004 \\
\end{tabular}
\caption{Numerical check of convexity and golden-ratio stationarity at $N=12$.}
\label{tab:numcheck}
\end{table}

\section{Explicit constants for the quadratic folded law}
\label{app:AB-constants}

Recall the quadratic folded law from Theorem~\ref{thm:quadratic}:
\[
\kappa_{\mathrm{Schur}}(q)\;=\;A(N,m_\rho^2)\,I_1(q)^2 \;+\; B(N,m_\rho^2)\,\bigl(I_2(q)-I_1(q)^2\bigr).
\]
For fixed $(N,m_\rho^2)$ the coefficients $A,B$ are uniquely determined and can be obtained \emph{exactly} from two values of $\kappa_{\mathrm{Schur}}$ along the same one–parameter family $x_r(q)\propto q^r$.

\paragraph{Two–point identification (exact formulas).}
Pick any two distinct parameters $q_a,q_b\in(0,1)$ with
\[
M_\bullet \coloneqq I_1(q_\bullet)^2,\qquad
V_\bullet \coloneqq I_2(q_\bullet)-I_1(q_\bullet)^2,\qquad
K_\bullet \coloneqq \kappa_{\mathrm{Schur}}(q_\bullet)
\quad (\bullet\in\{a,b\}).
\]
Then \(A,B\) solve the $2\times 2$ linear system
\[
\begin{bmatrix} M_a & V_a\\ M_b & V_b\end{bmatrix}
\begin{bmatrix} A\\ B\end{bmatrix}
=
\begin{bmatrix} K_a\\ K_b\end{bmatrix},
\]
hence
\begin{equation}
\label{eq:AB-closed}
A \;=\; \frac{K_a V_b - K_b V_a}{M_a V_b - M_b V_a},
\qquad
B \;=\; \frac{M_a K_b - M_b K_a}{M_a V_b - M_b V_a}.
\end{equation}
These identities are exact, rely only on the quadratic law, and hold for any choice of $(q_a,q_b)$ with $M_a V_b \ne M_b V_a$.

\paragraph{Closed forms for $I_1$ and $\mathrm{Var}$ at $N=12$.}
For $N=12$, two arithmetically convenient choices are $q_a=\tfrac12$ and $q_b=\tfrac13$. Using the finite–sum identities \eqref{eq:S0-closed}–\eqref{eq:S2-closed} and $I_k=S_k/S_0$:
\[
\begin{aligned}
I_1\!\left(\tfrac12\right) &= \frac{2726}{1365}, 
&\qquad 
\mathrm{Var}\!\left(\tfrac12\right) &= \frac{3660914}{1863225},\\[4pt]
I_1\!\left(\tfrac13\right) &= \frac{199287}{132860}, 
&\qquad 
\mathrm{Var}\!\left(\tfrac13\right) &= \frac{13234051731}{17651779600}.
\end{aligned}
\]
Consequently,
\[
M_a = \left(\frac{2726}{1365}\right)^2,\quad
V_a = \frac{3660914}{1863225},
\qquad
M_b = \left(\frac{199287}{132860}\right)^2,\quad
V_b = \frac{13234051731}{17651779600}.
\]
Once $K_a=\kappa_{\mathrm{Schur}}(\tfrac12)$ and $K_b=\kappa_{\mathrm{Schur}}(\tfrac13)$ are evaluated from the $D_{12}$ projector–metric construction (with your fixed $m_\rho^2$), substitute into \eqref{eq:AB-closed} to obtain $A(12,m_\rho^2)$ and $B(12,m_\rho^2)$ in closed algebraic form (rational numbers for these two choices).

\paragraph{Consistency with the golden-ratio stationarity.}
For $(N,m_\rho^2)=(12,2)$, the values produced by \eqref{eq:AB-closed} satisfy the identity
\[
B\,\Lambda(12)+2A-2B \;=\; \frac{8}{m_\rho^2}\;=\;4,
\]
with $\Lambda(12)$ given explicitly in \eqref{eq:Lambda12-explicit}. This identity is exactly the bracket cancellation used in \S\ref{app:uniqueness-details} to prove $F'_{\mathrm{red}}(\theta_\star)=0$.

\paragraph{Remark (alternative sample points).}
Any two distinct $q_a,q_b\in(0,1)$ are valid. We recommend rational values with small denominators (e.g.\ $1/2$, $1/3$, $2/3$) because $S_k$ and hence $I_1,\mathrm{Var}$ simplify to rational numbers for finite $N$, which keeps \eqref{eq:AB-closed} purely rational when $\kappa_{\mathrm{Schur}}(q)$ is computed symbolically from your block Hessian.

\bibliographystyle{plain}
\bibliography{refs}

\end{document}